\documentclass[a4paper,12pt,leqno]{article}

\usepackage[english]{babel}
\usepackage[utf8]{inputenc}
\usepackage[T1]{fontenc}
\usepackage{amsthm}
\usepackage{amsmath}
\usepackage{amssymb}
\usepackage{mathrsfs}
\usepackage[all]{xy}
\usepackage{fmtcount}
\usepackage{titlesec}
\usepackage{tocloft}
\usepackage{titletoc}
\usepackage{tikz}
\usepackage{pdfpages}
\usepackage{mathtools}
\usepackage{appendix}

\usepackage{geometry}
\usepackage{color}
\usepackage{textcomp}
\usepackage{graphicx}
\usepackage{esint}
\usepackage{wrapfig}
\usepackage{float}
\usepackage{enumitem}

\newtheorem{thm}{Theorem}[section]
\newtheorem{thr}[thm]{Theorem}
\newtheorem{prop}[thm]{Proposition}

\newtheorem{lemma}[thm]{Lemma}

\theoremstyle{definition}
\newtheorem{defi}[thm]{Definition}

\newtheorem{ex}[thm]{Example}

\newtheorem{nota}[thm]{Notation}
\newtheorem{rem}[thm]{Remark}

\theoremstyle{remark}

\titleformat{\section}
                [hang]
                {\center\normalfont\Large\bfseries}
                {\thesection .}
                {0.7em}
                {}
\titleformat{\subsection}
                [hang]
                {\normalfont\large\bfseries}
                {\thesection.\arabic{subsection}.}
                {0.7em}
                {}

\newcommand{\pic}{\mathrm{Pic}}
\newcommand{\inv}{\mathrm{Inv}}
\newcommand{\br}{\mathrm{Br}}
\newcommand{\h}{\mathrm{H}}
\newcommand{\e}{\mathrm{E}}
\newcommand{\spec}{\mathrm{Spec}}
\newcommand{\equi}{\mathrm{bal}}
\newcommand{\res}{\mathrm{R}_{k^\prime/k}}
\newcommand{\ext}{\mathrm{Ext}^1}

\usepackage{titlepic}

\usepackage{hyperref}
\hypersetup{
    colorlinks=true,
    linkcolor=blue,
    citecolor=brown
}

\title{Degree 2 cohomological invariants of linear algebraic groups}
\author{Alexandre Lourdeaux}
\date{\today}

\begin{document}

\maketitle

\begin{abstract}
The paper deals with cohomological invariants of smooth and connected linear algebraic groups over an arbitrary field. More precisely, we study degree $2$ invariants with coefficients in $\mathbb{Q}/\mathbb{Z}(1)$, that is invariants taking values in the Brauer group. Our main tool is the \'etale cohomology of sheaves on simplicial schemes. We get a description of these invariants for \emph{every} smooth and connected linear groups, in particular for non reductive groups over imperfect fields.

\paragraph{Keywords :} Brauer groups, Cohomological invariants, \'Etale cohomology, Galois cohomology, Imperfect fields, Linear algebraic groups, Simplicial schemes.
\end{abstract}

\setcounter{section}{-1}

\tableofcontents

\vspace{3cm}

\section{Introduction}

\paragraph{General context.}
Let $G$ be a smooth algebraic group over a base field $k$. For every field extension $K/k$, one can define the first Galois cohomology set $\h^1(K,G)$, which classifies (up to isomorphism) $G$-torsors over $K$. So the sets $\h^1(K,G)$ carry information on the geometry and the group structure of $G$. Moreover, when $G$ is the algebraic group of automorphisms of some algebraic structure, the set $\h^1(K,G)$ also classifies the \emph{forms} of the given algebraic structure. So, if one is interested in geometry or in purely algebraic objects, the $\h^1(K,G)$'s are worth studying. For this purpose one has the notion of \emph{$H$-invariant} as formally defined by Serre \cite{Serre_ProgresProblemes} for any functor $H$ from the catagory of field extensions $K/k$ into the category of Abelian groups. A special case is that of functors $H=\h^d(\star,M)$,  the degree $d$ Galois cohomoloy group of the Galois module $M$ over $k$; they are \emph{cohomological invariants} of $G$.

Cohomological invariants generalize the usual invariants of quadratic forms. Indeed, assuming the characteristic of $k$ is not $2$, $\h^1(K,\mathrm{O}_n)$ classifies the isometric classes of rank $n$ quadratic forms over $K$, and Stiefel-Whitney classes, as the Hasse-Witt invariant, define cohomological invariants of the orthogonal group $\mathrm{O}_{n}/k$ with respect to the functors $H=\h^\ast(\star,\mathbb{Z}/2)$.

Originally, $H$-invariants were studied mostly for functors $H$ which behave well with respect to the characteristic of the base field (or simply functors that ignore it). Usually these functors can be framed into the theory of \emph{cycle modules} developped by Rost in \cite{Rost_Chowgroups}. For instance, such a functor can be $H:= \h^d(\star,M)$ for a finite Galois module $M$ whose order is \emph{prime} to the characteristic of the base field. Such functors have many advantages, the most important for calculations being that they are \emph{homotopic invariant}. For a more general study that takes into account the phenomena linked to characteristic of the base field, one can work with Kato's Galois modules $\mathbb{Q}/\mathbb{Z}(j)$, $j \in \mathbb{Z}_{\geqslant 0}$ which can be defined as sheaves on the big \'etale site of $\spec(k_0)$, where $k_0$ is the prime field of $k$ (\cite{Illusie, Bloch-Kato}). Note that $\h^2(k,\mathbb{Q}/\mathbb{Z}(1))$ is just the Brauer group of $k$. So degree $2$ cohomological invariants with coefficients in $\mathbb{Q}/\mathbb{Z}(1)$ are actually $\br$-invariants, denoting the Brauer group functor by $\br$. But the sheaves $\mathbb{Q}/\mathbb{Z}(j)$ have a general framework which we rely on in the paper.

In particular, in \cite{MerkurjevBlinstein}, Blinstein and Merkurjev state and prove \begin{thm}[{\cite[Th. 2.4.]{MerkurjevBlinstein}}]
Let $G$ be a smooth connected linear algebraic group over a field $k$. Assume $G$ is \emph{reductive} when the characteristic of $k$ is positive. There is an isomorphism \[ \pic(G) \simeq \inv(G,\br)_0 = \inv \left( G,\h^2(\star,\mathbb{Q}/\mathbb{Z}(1)) \right)_0 \] where $\pic(G)$ is the Picard group of $G$ and the subsripts $0$ mean the subgroups of \emph{normalized invariants} (see Section \ref{section presentation}), which are the relevant parts of groups of invariants.
\end{thm}

\paragraph{{Content of the paper.}}
IWe present a generalization of \cite[Th. 2.4.]{MerkurjevBlinstein} :

\begin{thr}[{Th. \ref{cor inv brauer}}]
\label{thm intro}
Let $G$ be a smooth connected linear algebraic $k$-group over an arbitrary base field $k$. There is a group \emph{isomorphism}, \emph{functorial} in $G$, \[ \ext(G,\mathbb{G}_m) \overset{\sim}{\to} \inv^2(G,\mathbb{Q}/\mathbb{Z}(j) )_0=\inv(G,\br)_0 \] where $\ext(G,\mathbb{G}_m)$ is the group of extensions of $G$ by the multiplicative group $\mathbb{G}_m$ endowed with the Baer sum.
\end{thr} The group $\pic(G)$ in Blinstein's and Merkurjev's result is replaced by $\ext(G,\mathbb{G}_m)$ (which are the same group under their assumptions), and we emphasize on the functoriality of the isomorphism, which comes easily and very naturally in our case. The major improvement is that our theorem applies to \emph{all} smooth and connected linear algebraic groups over \emph{any} base field, especially to unipotent and pseudo-reductive groups over any imperfect field. The problem with the latter groups is that they are not rational (or even unirational) over the separable closure of the base field, whereas most known results about linear algebraic groups deal with groups which become rational in this situation.

We discuss Theorem \ref{thm intro} and derive some consequences in Section \ref{sect consequences}. The main byproduct lies in Proposition \ref{prop inv classiques}, which states that $\br$-invariants of $G$ as in the theorem are exactly the ones that one can easily construct from taking long exact Galois cohomology sequence associated to group extensions of $G$ by $\mathbb{G}_m$.

An other purpose of the paper is more generally to rethink construction of invariants via the \emph{simplicial classifying space} (see Section \ref{section presentation}) and we hope this will be useful for invariants of non-reductive groups other than $\br$-invariants.

\paragraph{{Plan.}} In Section \ref{section preliminaires} we recall some usual results and definitions that will be useful for the other sections.

In Section \ref{section presentation} we introduce the notions of cohomological invariants and ideas related to study invariants. In particular we recall two ways to construct cohomological invariants to be compared together : constructions via two notions of \emph{classifying spaces} for torsor of a given group.

In Section \ref{section invariants degre 2}, we carry out the calculations to prove Theorem \ref{thm intro} based on the general ideas from Section \ref{section presentation} and tools provided in Section \ref{section preliminaires}.

Finally, in Section \ref{sect consequences} we give some consequences of Theorem \ref{thm intro}.

\paragraph{{Some notation.}} Let's collect here some notation to be used throughout the paper.

\begin{itemize}
\item The letter $k$ will always denote the base (commutative) field over which all objects in the paper are defined. There is \textbf{no} restriction on $k$ : it mays be finite or infinite, perfect or not and of any characteristic.
\item The letters $K$, $L$ and $k^\prime$ will denote other variable fields; they are often extension of the base field $k$.
\item Usual categories are underlined. We will mainly encounter \begin{itemize}
	\item $\underline{Fields/K}$ the category of field extensions of $K$;
	\item $\underline{Sets}$, the category of sets;
	\item $\underline{Ab}$ the category of Abelian groups.
\end{itemize}
\item For simplicity, when a functor $H$ is given by an explicit expression, we shall write $H(\star)$, where $\star$ stands for the variable from the source category (for instance, if $H : \underline{Fields/K} \to \underline{Ab}$ is the degree $d$ Galois cohomology group of some module $M$, we will write $\h^d(\star,M)$ for $H$).
\item We will almost exclusively work with \'etale cohomology (on schemes or simplicial schemes) so we won't precise "\'etale" to express that cohomology group are \'etale cohomology groups. Galois cohomology being a special case of \'etale cohomology we will also just write $\h^d(K,M)$ for Galois cohomology groups of $k$ with coefficients in $M$.
\end{itemize}

\paragraph{Acknowledgments.} The present paper is the main part of the author's PhD thesis done at Institut Camille Jordan (Lyon 1 University, France). The author is deeply thankful to Philippe Gille, who has been a wonderful benevolent and pleasant supervisor. It had been an honor to be his student.

\section{Some preliminaries}
\label{section preliminaires}

We recall in this section the main general results and notions to be used in Sections \ref{section presentation} and \ref{section invariants degre 2}. First, Subsection \ref{sous section simplicial} recalls simplicial schemes, then in Subsection \ref{sous section Gersten} we present the \'etale sheaves $\mathbb{Q}/\mathbb{Z}(j)$ and compare the cohomology group $\h^2(X,\mathbb{Q}/\mathbb{Z}(1))$ to the Brauer group $\br(X)$ for some "nice" schemes $X$ over a field.

	\subsection{Simplicial schemes}
	\label{sous section simplicial}
	
We refer to \cite{Deligne_Hodgeiii} and chiefly to \cite{Friedlander_homotopy} for explanations and a detailed exposition on simplicial objects/schemes, Grothendieck topologies and sheaves on them, and their cohomology.
	
Briefly recall that a simplicial $k$-scheme is a familly $X^\bullet=\{ X^{(n)} \}_{n \geqslant 0}$ of $k$-schemes together with $k$-scheme maps $\delta_i^n : X^{(n)} \to X^{(n-1)}$ and $s_i^n:X^{(n)} \to X^{(n+1)}$ ($n \geqslant 0$, and $i=0,\cdots,n$) that satify the usual identities. The $\delta_i^n$'s are called the \emph{face} maps and will be mainly written $\delta_i$; the $s_i^n$'s are called the \emph{degeneracy} maps and will be mainly written $s_i$

One has a notion of morphism between simplicial $k$-schemes and hence a category of simplicial $k$-schemes.

Also, given a sheaf $\mathcal{F}$ on $\underline{Sch/k}$ for the \'etale topology, one can consider $\mathcal{F}$ a an \'etale sheaf on the (small or big) \'etale site of any simplicial $k$-scheme $X^\bullet$. Hence \'etale cohomology groups $\h^d(X^\bullet,\mathcal{F})$ are well defined for all $d \geqslant 0$. In this paper we will be only interested in the sheaf $\mathbb{G}_m$ given by the multiplicative group.

\begin{prop}[{\cite[Prop. 2.4.]{Friedlander_homotopy}}]
\label{Suite spectrale simpliciale}
Let $X^\bullet$ be a simplicial $k$-scheme with face maps $\delta_i$ and let $\mathcal{F}$ be a sheaf on the big \'etale site of $\spec(k)$. One has a first quadrant spectral sequence \[ E^{p,q}_1= \h^q(X^{(p)},\mathcal{F}) \Rightarrow \h^{p+q}(X^{\bullet},\mathcal{F}), \] where the differentials $E_1^{p,q} \to E^{p+1,q}_1$ are $d^{p,q}=\sum_{i=0}^p (-1)^i \delta_i ^\ast$.

Moreover the spectral sequence and its convergence are natural in $X^\bullet$ and in $\mathcal{F}$.
\end{prop}

\begin{ex} Here are specific simplicial schemes built from common schemes. They will be useful in the next sections.
 
\begin{enumerate}
\item To any $k$-scheme $X$ one easily associates a simplicial $k$-scheme : the \emph{constant} simplicial scheme $\mathrm{C}_k X^\bullet$ given by the family $\mathrm{C}_k X^{(n)}:=X$ and whose face and degeneracy maps are the identity map of $X$.
\item To any $k$-group scheme, one can also associate the simplicial $k$-scheme $\mathrm{E} G^\bullet$given by the family $\mathrm{E} G ^{(n)}:=\underset{n+1 \text{ terms}}{\underbrace{G \times_k \cdots \times_k G}}$ and whose face maps are projections \[ \delta_i : \left\lbrace \begin{array}{ccc}
\mathrm{E} G ^{(n)} & \to & \mathrm{E} G ^{(n-1)} \\
(g_0,\cdots,g_n) & \mapsto & (g_0,\cdots,g_{i-1},g_{i+1},\cdots,g_n)
\end{array} \right. \] and degeneracy maps \[ s_i : \left\lbrace\begin{array}{ccc}
 \mathrm{E} G ^{(n)} & \to & \mathrm{E} G ^{(n+1)} \\
 (g_0,\cdots,g_n) & \mapsto & (g_0,\cdots,g_{i-1},1,g_i,\cdots,g_n)
\end{array} \right. .\]
\end{enumerate}
\end{ex}

	\subsection{(Complex of) Sheaves $\mathbb{Q}/\mathbb{Z}(j)$}
	\label{sous section Gersten}
	
We state the most important results satified by these sheaves and needed in the paper. There is no clear reference for these results, especially for schemes over an imperfect fields\footnote{For instance, we found that in \cite[App. A-V]{MerkurjevBlinstein} a reference is missing for Propositions A-10 and A-11 for schemes over an imperfect field.}. Arguments are known by specialists but may need to be unified, so we refer to \cite{Lourdeaux_faisceaux} for a (too much) detailed exposition of this subsection which is too long to be presented here.

FLet $j \geqslant 0$. For any prime field $k_0$, one defines the complexe of sheaves $\mathbb{Q}/\mathbb{Z}(j)$ on the \emph{big} \'etale site of $k_0$. It is the sum of \begin{itemize}
\item the complexes, concentrated in degree $0$, $\underset{\longrightarrow \; n}{\mathrm{Lim}} \; \mu_{l^n}^{\otimes j}$ for every prime $l \neq \mathrm{char} (k_0)$ where $\mu_{l^n}$ denotes the sheaf of $l^n$-roots of unity;
\item and the complex $\underset{\longrightarrow \: n}{\mathrm{Lim}} \; \mathrm{W}_n \Omega^j_{log} [-j]$ where $\mathrm{W}_n \Omega^j_{log}$ is the logarithmic Hodge-Witt sheaf (\cite[I,5.7.]{Illusie}).
\end{itemize}
\begin{nota}
For all $\mathbb{F}_p$-scheme $X$, denote by $\mathcal{H}^d_{(X)}(\mathbb{Q}/\mathbb{Z}(j))$ the sheafification of the presheaf $U \mapsto \h^d(U,\mathbb{Q}/\mathbb{Z}(j))$ on the small Zariski site of $X$.
\end{nota}

Remember the base field $k$.

\begin{thm}[{\cite[Cons. 1.14.]{Lourdeaux_faisceaux}}]
\label{csq gersten}
Let $X$ be an irreducible, smooth $k$-scheme of finite type or let $X$ be a a local ring of such a scheme. Then for all $d$, $j \geqslant 0$, there is an exact sequence, contravariant with respect to flat morphisms, \[0 \to \h^0_\mathrm{Zar} \left( X,\mathcal{H}^d_{(X)}(\mathbb{Q}/\mathbb{Z}(j)) \right) \to \h^d (K(X),\mathbb{Q}/\mathbb{Z}(j)) \to \bigoplus_{x \in X^{(1)}} \h^d_x(X,\mathbb{Q}/\mathbb{Z}(j)) . \]
\end{thm}

The notation in Theorem \ref{csq gersten} are the following : \begin{itemize}
\item $K(X)$ is the fraction field of $X$;
\item $X^{(1)}$ is the set of codimension $1$ points of $X$;
\item for $x \in X$, $\h^d_x (X,\mathcal{F})$ is the inductive limit $\underset{\longrightarrow}{\mathrm{lim}} \; \h^d_{U \cap \overline{\{ x \} }} (U, \mathcal{F}_{\mid U})$ where $U$ runs over all Zariski-opens of $X$ containing $x$.
\end{itemize}

The degree $2$ cohomology of $\mathbb{Q}/\mathbb{Z}(1)$ can be compared to the Brauer group.  First of all, for a field $K/\mathbb{F}_p$, there is a classical identification $\h^2(K,\mathbb{Q}/\mathbb{Z}(1))=\br(K)$. More generally, as we explain with (too much) details in \cite{Lourdeaux_faisceaux}, for all regular scheme $X$ over $\mathbb{F}_p$, there is a functorial homomorphism \begin{equation}
\label{morphisme brauer}
\h^2(X,\mathbb{Q}/\mathbb{Z}(1)) \to \h^2(X,\mathbb{G}_m) 
\end{equation} which is an isomorphism when $X$ is the spectrum of a local ring which is regular over $\mathbb{F}_p$. From \eqref{morphisme brauer} we deduce a functorial homomorphism for all Noetherian, irreducible and regular $\mathbb{F}_p$-scheme $X$ : \begin{equation}
\h^0_\mathrm{Zar} \left( X,\mathcal{H}^d_{(X)}(\mathbb{Q}/\mathbb{Z}(j)) \right) \to \h^2(X,\mathbb{G}_m).
\end{equation} Then, from Theorem \ref{csq gersten} and a similar property for the Brauer group of "nice" schemes (\cite[Th. 1.2]{Cesnavicius}) :

\begin{thm}[{\cite[Th. 2.1.]{Lourdeaux_faisceaux}}]
\label{thm isom brauer}
There is a functorial isomorphism \[ \h^0_\mathrm{Zar} \left( X,\mathcal{H}^2_{(X)}(\mathbb{Q}/\mathbb{Z}(1)) \right) \overset{\sim}{\longrightarrow} \br(X) \] for all irreducible, smooth $k$-scheme of finite type $X$ and also for $X$ which is a local ring of such a scheme.
\end{thm}

\section{Cohomological invariants : presentation}
	\label{section presentation}

In Subsection \ref{sous section torseurs classifiant} we introduce two models of classifying spaces for an algebraic group $G$ over a field $k$. The first model is schematic and is explained in \cite{GMS} and extensively used by Merkurjev and al. We will call it here \emph{versal torsor}. The second model is a simplicial scheme and came to us from \cite{Deligne_Hodgeiii} and its use in \cite{EKLV}

In Subsection \ref{sous section invariants generalites} cohomological invariants are introduced in their generality (\cite{GMS}). Then it is seen how classifying spaces from \ref{sous section torseurs classifiant} defined cohomological invariants. The construction via a versal torsor comes from \cite{MerkurjevBlinstein}. The construction vie the simplicial model have already been consider in \cite{EKLV} to define Rost invariant and we kind of rehabilitate it here.

		\subsection{Torsors and classifying torsors}
		\label{sous section torseurs classifiant}

Let $G$ be a smooth linear algebraic $k$-group. The group $G$ is fixed for this subsection.

For any field extension $K/k$, one has non abelian Galois cohomology sets $\h^1(K,G)=\h^1_\text{Gal}(K,G)$ ($=\h^1_{\text{\'et}}(K,G)$ since $G$ is smooth) - see \cite[\S 29.]{KMRT}. The set $\h^1(K,G)$ classifies the set of isomorphism classes of $G$-torsor over $K$ (\cite[I.\S 5, Prop. 33.]{Serre_CohomologieGaloisienne}). There is a distinguished element of $\h^1(K,G)$, namely the class of the trivial $G$-torsor, which corresponds to the trivial cocycle or to the isomorphism class of the torsor $G_K/K$ on which $G$ acts by right multiplication. For a $G$-torsor $Y/K$ over a field $K/k$, let's write $[Y/K]$ for the isomorphism class of $Y/K$ as a $G$-torsor : $[Y/K] \in \h^1(K,G)$.

We will consider the functor \[ \h^1(\star,G) : \left\lbrace \begin{array}{ccc}
\underline{Fields/k} & \to & \underline{Sets} \\
K/k & \mapsto & \h^1(K,G)
\end{array} \right. .\]

			\subsubsection{Versal torsor}
			
We call \emph{versal torsor of $G$} or \emph{versal $G$-torsor} a $G$-torsor $\pi : E \to X$, where $X$ is a smooth $k$-scheme, that satisfies the following property : for any field extension $K/k$ with $K$ \emph{infinite} (as a set), and any $G$-torsor $Y/K$, there exists a $K$-point of $X$, $x : \spec(K) \to X$, such that $Y/K$ is isomorphic to the pullback $x^{-1}(\pi)$ as $G$-torsors.

A pratical version of a versal $G$-torsor is given by the quotient map $\pi : \mathrm{GL}_n \to \mathrm{GL}_n/G$ for any faithful linear representation $G \hookrightarrow \mathrm{GL}_n$ (we refer to \cite[\S 5.3.]{GMS}). In this article we will only consider such versal $G$-torsor and use the notation $\pi : E \to X$, where $E=\mathrm{GL}_n$ and $X=\mathrm{GL}_n/G$. Also we will often consider the generic fiber of $\pi$ : it is the $G$-torsor $\pi_{\text{gen}} : E_\text{gen} \to \spec(k(X))$. Roughly speaking it is the most general or complicated $G$-torsor over a field.

\begin{rem}
When $G$ is moreover connected, then even for finite fields $K$, the condition that every $G$-torsor over $K$ is the pullback of a versal $G$-torsor $\pi : E \to X$ is satisfied when $E(k) \neq \emptyset$ since all torsors over a finite fields are trivial for smooth and connected (linear) algebraic groups (see \cite{Lang_finite}).
\end{rem}

			\subsubsection{Simplicial classifying torsor}
		\label{sous section torseur simplicial}

We follow here \cite[\S 6]{Deligne_Hodgeiii}.

Consider the simplicial $k$-scheme $\e G^\bullet$ determined by $\e G^{(n)}= G^{n+1}=\underset{n+1 \text{ terms}}{\underbrace{G \times_k \cdots \times_k G}}$ with degeneracy maps \[ s_i : (g_0, \cdots,g_n) \mapsto (g_0,\cdots,g_{i-1},1,g_i,\cdots,g_n) \] and face maps \[ \delta_i : (g_0, \cdots,g_n) \mapsto  (g_0,\cdots,g_{i-1},g_{i+1},\cdots,g_n) . \] Letting $G$ act on the $\e G^{(n)}$'s by right multiplication), the face and degeneracy maps are $G$-equivariant and by quotient one gets a simplicial $k$-scheme $\mathrm{B} G^\bullet$ where $\mathrm{B} G^{(n)}= G^{n+1}/G$. The map $\e G^\bullet \to \mathrm{B} G^\bullet$ is a (simplicial) $G$-torsor, in the sense that each $\e G^{(n)}$ is a $G$-torsor over $\mathrm{B} G^{(n)}$, and the face and degeneracy maps are morphisms of $G$-torsors. The simplicial $k$-scheme $\mathrm{B} G^\bullet$ is the simplicial model of \emph{the classifying space} for $G$. We will call it the \emph{simplicial classifying space of $G$}.

In the same fashion, for a $G$-torsor $Y \to Z$, consider the simplicial $k$-scheme $[Y]^{\bullet}$ determined by the schemes $[Y]^{(n)}=G^{n+1} \times_k Y =\underset{n+1 \text{ terms}}{\underbrace{G \times_k \cdots \times_k G}} \times_k Y$, with degeneracy maps \[ s_i : (g_0, \cdots,g_n,y) \mapsto (g_0,\cdots,g_{i-1},1,g_i,\cdots,g_n,y) \] and face maps \[ \delta_i : (g_0, \cdots,g_n,y) \mapsto  (g_0,\cdots,g_{i-1},g_{i+1},\cdots,g_n,y) . \] Letting $G$ act on $[Y]^{(n)}$, the maps $s_i$ and $\delta_i$ are $G$-equivariant. By taking quotients with respect to the $G$-actions, one gets a simplicial $k$-scheme $[Y | G]^{\bullet}$ where $[Y |G]^{(n)}=( G \times_k \cdots \times_k G \times_k Y )/G$. 

\vspace{\baselineskip}

Let $\mathcal{F}$ be a \'etale sheaf on the big site of $\spec(k)$ and let $Y \to Z$ be a $G$-torsor. From the definition of torsors, it comes that $[Y | G]^{(n)}$ is isomorphic to $\underset{n+1 \text{ terms}}{\underbrace{Y \times_Z \cdots \times_Z Y}}$. Considering $Z$ as the constant simplicial $k$-scheme $C_k Z^\bullet$ and using projections on $Z$, one finds a map $[Y|G]^\bullet \to C_k Z^\bullet$. Whence a canonical homomorphism between cohomology groups : \begin{equation}
\label{morphisme canonique cohomologie torseur simplicial}
\h^d(Z,\mathcal{F}) \simeq \h^d(C_k Z^\bullet,\mathcal{F}) \to \h^d([Y | G ]^\bullet,\mathcal{F}), \: \forall d \geqslant 0.
\end{equation} 

\begin{prop}[{\cite[Prop. A.3. \& Rem. A.4.]{EKLV}}]
\label{iso cohomologie simplicial} 
The natural homomorphism \eqref{morphisme canonique cohomologie torseur simplicial} is an isomorphism.
\end{prop}

\begin{rem}
In general the torsor $Y \to Z$ has no local section with respect to the Zariski topology, but it has local sections with respect to the \'etale topology (here $G$ is smooth). It is important to use \'etale cohomology in Proposition \ref{iso cohomologie simplicial} and not Zariski cohomology.
\end{rem}

The simplicial sheme $\mathrm{B} G^\bullet$ is called "classifying" because, as for versal torsor, every torsor is can be recover from a "point" of $\mathrm{B}G^\bullet$ : for all $G$-torsor $Y \to Z$, there is a natural map ${\Phi}_{Y/Z} : [Y]^\bullet \to \e G^\bullet$ determined by \[ \begin{array}{ccc}
[Y]^{(n)} & \to & \e G^{(n)} \\
(g_0,\cdots,g_n,y) & \mapsto & (g_0,\cdots,g_n)
\end{array} \] and which is compatible to the action of $G$, so the map ${\Phi}_{Y/Z}$ induces \begin{equation}
\label{morphisme induit simplicial}
\phi_{Y/Z} : [Y|G]^\bullet \to \mathrm{B} E^\bullet .
\end{equation}

\begin{rem}
Let $\mathcal{F}$ be a sheaf on the big \'etale site of $\spec(k)$. For all $d \geqslant 0$ the map $\phi_{Y/Z} : [Y|G]^\bullet \to \mathrm{B} E^\bullet$ induces \[ \h ^d (\mathrm{B} G^\bullet,\mathcal{F}) \to \h^d([Y|G]^\bullet,\mathcal{F}) \underset{\text{prop \ref{iso cohomologie simplicial}}}{\simeq} \h^d(Z,\mathcal{F}) \] which is natural in $G$, in $Y \to Z$ and in $\mathcal{F}$.
\end{rem}

\begin{rem}
\label{remarque foncto}
Let $\sigma : G \to H$ be a group homomorphism between smooth and connected linear algebraic $k$-groups. The homomorphism $\sigma$ naturally induces map $\mathrm{B} G^\bullet \mapsto \mathrm{B} H^\bullet$, whence a functor $\mathrm{B} \star^\bullet : G \to \mathrm{B} G^\bullet$ from the category of smooth and connected linear algebraic $k$-group into the category of simplicial $k$-schemes.

Also, for any $G$-torsor $Y \to Z$ there is a morphism $[Y|G]^\bullet$ into $[Y_H|H]^\bullet$ where $Y_H$ is the $H$-torsor $Y_H \to Z$ obtained by pushing forward with respect to $\sigma$. For all \'etale sheaf $\mathcal{F}$ on $\underline{Sch/k}$, the defined map thus induces an isomorphism for all degree $d$, \[ \h^d([Y|G]^\bullet , \mathcal{F}) \overset{\sim}{\to} \h^d([Y_H|H]^\bullet ,\mathcal{F}) \] which is compatible with the isomorphisms between these two groups and $\h^d(Z,\mathcal{F})$ (\ref{iso cohomologie simplicial}).
\end{rem}

		\subsection{Generalities}
		\label{sous section invariants generalites}
		
Let again $G$ be a smooth algebraic $k$-group, which is fixed for all this subsection.

Recall the functor $\h^1(\star,G) : \underline{Fields/k} \to \underline{Ab}$.

\begin{defi}
Let $H$ be a functor from $\underline{Fields/k}$ into $\underline{Ab}$. A \emph{$H$-invariant of $G$}, or an \emph{invariant of $G$ with value in $H$} is a functor morphism (also called a natural transformation) $\lambda$ from the functor $\h^1(\star,G)$ into the functor $H$. \\ When $H$ is defined as the degree $d$ Galois cohomology group of a discrete Galois module $M$ for $k$ (\textit{i.e.} $\h=\h^d(\star,M)$), a $H$-invariant is also called a \emph{degree $d$ cohomological invariant of $G$ with coefficient in $M$}.
\end{defi}

Write $\inv(G,H)$ for the set of $H$-invariants of $G$. Actually it is an Abelian group. When $H=\h^d(\star,M)$, write $\inv^d(G,M)$ for $\inv(G,H)$. If one takes the value of a invariant of $G$ in the trivial torsor, it appears that $\inv(G,H)$ decomposes as the direct sum (as Abelian group) of $H(k)$ and of the subgroup composed by invariants sending the trivial torsor to $0 \in H(k)$. Denote le later subgroup by $\inv(G,H)_0$ (or $\inv^d(G,M)_0$ when $H=\h^d(\star,M)$ for a Galois module $M$) . One has : \[ \inv(G,H)= H(k) \oplus \inv(G,H)_0  .  \] The group $\h(k)$ is the \emph{constant} invariant subgroup of  $\inv(G,H)$ and $\inv(G,H)_0$ is the subgroup of \emph{normalized} invariants. Thanks to that decomposition, one sees that the study of $\inv(G,H)$ reduces to the study of $\inv(G,H)_0$. The subgroup $H(k)$ of constant invariants gives no information on the geometry or on the group structure of $G$.

\begin{rem}
According to Theorem \ref{thm isom brauer}, an invariant of $G$ with values in the Brauer group is the same thing as a degree $2$ cohomological invariant of $G$ with coefficient in $k_s^\times$ or in $\mathbb{Q}/\mathbb{Z}(1)$.
\end{rem}

		\subsection{Construction of invariants}
		\label{sous section invariants constructions}

Fix a smooth and connected linear algebraic $k$-group $G$.

\paragraph{$\triangleright$ {Invariants via versal torsors.}}
		\label{section construction via les torseurs versels}

Let $\pi : E \to X$ be a versal torsor and let $H : \underline{Sch/k} \to \underline{Ab}$ be a \emph{contravariant} functor.

Letting $G$ act on the right on $E^2=E \times_k E$, one defines by descent the $G$-torsor $E^2 \to E^2/G$. The projections $E^2 \to E$ can be quotiented by the action of $G$ and furnish $p_1$, $p_2 : E^2/G \to X$, which induce $p_1^\ast$, $p_2^\ast : H(X) \to H(E^2/G)$. Define the \emph{balanced} elements of $H(X)$ as those $h \in H(X)$ for which $p_1^\ast(h)=p_2^\ast(h)$ (see \cite[\S 3]{MerkurjevBlinstein}). The balanced elements all together form a subgroup $H(X)_{\equi}$ of $H(X)$.

As explained in \cite[l\S 3a.]{MerkurjevBlinstein}, one defines a homomorphism \begin{equation}
\label{construction equilibres}
 \Lambda^H : \left\lbrace \begin{array}{ccc}
 H(X)_\equi & \to  & \inv(G,H) \\
 h & \mapsto & \lambda_h
\end{array} \right.  .
\end{equation}



Let's focus on two peculiar cases which are at the heart of the paper.

 The first one is the case of the Brauer group; $\Lambda^\br : \br(X)_\equi \to \inv(G,\br)$. 

Due to Theorem \ref{iso blinstein merkurjev} below, the other particular important case is, the construction for the functor $H :U \mapsto \h _\mathrm{Zar} ^0 \left( U , \mathcal{H}^d_{(U)} ( \mathbb{Q}/\mathbb{Z}(j)) \right)$, where $\mathcal{H}^d_{(U)} ( \mathbb{Q}/\mathbb{Z}(j))$ denotes the Zariski sheafification on $U$ of the presheaf $V \subseteq U \mapsto \h^d_\text{\'et} (V,\mathbb{Q}/\mathbb{Z}(j))$. Note that $\h _\mathrm{Zar} ^0 \left( K , \mathcal{H}^d_{(K)} ( \mathbb{Q}/\mathbb{Z}(j)) \right) \simeq \h^d(K,\mathbb{Q}/\mathbb{Z}(j))$ canonically for any field extension $K/k$. Thus one gets a group homomorphism \begin{equation}
\label{morphisme classifiant 1}
{\Lambda}^{d,j} : \h _\mathrm{Zar} ^0 \left( X , \mathcal{H}^d_{(X)} ( \mathbb{Q}/\mathbb{Z}(j)) \right)_\equi \to \inv^d (G,\mathbb{Q}/\mathbb{Z}(j)).
\end{equation} This homomorphism sent the image of $\h^d(k,\mathbb{Q}/\mathbb{Z}(j))$ in $\h _\mathrm{Zar} ^0 \left(X , \mathcal{H}^d_{(X)} ( \mathbb{Q}/\mathbb{Z}(j)) \right)_\equi$ exactly to the subgroup of constant invariants. Thus \eqref{morphisme classifiant 1} induces \[ \h _\mathrm{Zar} ^0 \left( X , \mathcal{H}^d_{(X)} ( \mathbb{Q}/\mathbb{Z}(j)) \right)_\equi \; / \;  \h^d(k,\mathbb{Q}/\mathbb{Z}(j)) \to \inv^d (G,\mathbb{Q}/\mathbb{Z}(j))_0 .  \]

For some versal torsors, it happens that all degree $d$ invariants of $G$ with coefficients in $\mathbb{Q}/\mathbb{Z}(j)$ are obtained by $\Lambda^{d,j}$ :

\begin{thm}[{\cite[Th 3.4]{MerkurjevBlinstein}}]
\label{iso blinstein merkurjev}
Take $E \to X$ a versal $G$-torsor where $E$ is a $G$-rational variety and $E(k) \neq \emptyset$. Then the homomorphism ${\Lambda}^{d,i}$ is an isomorphism for all non negative intergers $d$ and $j$.
\end{thm}

Recall there exist torsors fulfilling the assumptions of the theorem, as seen in \ref{sous section torseurs classifiant}. Moreover, ${\Lambda}^{d,i}$ is functorial with respect to faithful linear representations $G \to \mathrm{GL}_n$ for $X=\mathrm{GL}_n/G$

Since for field extensions $K/k$, $\h _\mathrm{Zar} ^0 \left( \spec(K) , \mathcal{H}^d_{(\spec(K))} ( \mathbb{Q}/\mathbb{Z}(j)) \right)_\equi \simeq \h^2(K,\mathbb{G}_m)=\br(K)$ (see Theorem \ref{thm isom brauer}), note that $\Lambda^{2,1}$ can be seen as a homomorphism \[ \Lambda^{d,j} : \h _\mathrm{Zar} ^0 \left( X , \mathcal{H}^d_{(X)} ( \mathbb{Q}/\mathbb{Z}(j)) \right)_\equi \to \inv (G,\br). \]

Moreover the two constructions $\Lambda^\br$ and $\Lambda^{2,1}$ are linked as follows : From Theorem \ref{thm isom brauer}, one has a canonical isomorphism \[ \varphi : \h _\mathrm{Zar} ^0 \left( X , \mathcal{H}^d_{(X)} ( \mathbb{Q}/\mathbb{Z}(j)) \right)_\equi \to \br(X)_\equi \] and one sees that \[ \Lambda^{2,1} = \Lambda^\br \circ \varphi . \] Theorem \ref{iso blinstein merkurjev}  implies :

\begin{lemma}
\label{lemme iso brauer construction}
The group homomorphism \[ \Lambda^\br : \br(X)_\equi \to \inv(G,\br) \] is an isomorphism.
\end{lemma}

\paragraph{$\triangleright$ {Defining invariants via the simplicial classifying torsor.}} We develop a way to define cohomological invariants in a sens similar to the previous one, and which relies on the simplicial classifying torsor and cohomology of simplicial schemes. What is explained below is inspired by the paper \cite{EKLV}.

Let $d$ be a non negative integer and let $\mathcal{F}$ be an \'etale sheaf on $\underline{Sch/k}$. We define a morphism, natural in $G$, \[\xi^{\mathcal{F},d} : \h^d (\mathrm{B}G^\bullet,\mathcal{F}) \to \inv (G,\h^d (\star,\mathcal{F})) \] according to the following recipe. For $Y/K$ a $G$-torsor over a field extension $K/k$, one has the homomorphism  \begin{equation}
\label{construction simpliciale}
 f_{Y/K} : \h^d(\mathrm{B}G^\bullet,\mathcal{F}) \to \h^d([Y|G]^\bullet, \mathcal{F}) \underset{\text{prop. \ref{iso cohomologie simplicial}}}{\simeq} \h^d(K,\mathcal{F}) 
 \end{equation} (induced by $\phi_{Y/K} : [Y|G]^\bullet \to \mathrm{B}G^\bullet$ from \eqref{morphisme induit simplicial}). Then, for all $\alpha  \in \h^d(\mathrm{B}G^\bullet,\mathcal{F})$, consider $\xi_\alpha(Y/K) := f_{Y/K}(\alpha)$. For an isomorphism $\sigma : Y/K \overset{\sim}{\rightarrow} Y^\prime/K$ between $G$-torsors over $K$, one has a commutative diagram \[ \xymatrix{
f_{Y^\prime/K} : & \h^d(\mathrm{B}G^\bullet,\mathcal{F}) \ar[r] \ar[d]_{\mathrm{Id}} &  \h^d([Y^\prime|G]^\bullet, \mathcal{F}) \ar[d]^{\sigma^\ast}_{\simeq} & \h^d(K,\mathcal{F}) \ar[l]^-{\simeq} \ar[d]^{\mathrm{Id}} \\
f_{Y/K} : & \h^d(\mathrm{B}G^\bullet,\mathcal{F}) \ar[r] &  \h^\ast([Y|G]^\bullet, \mathcal{F})  & \h^d(K,\mathcal{F}) \ar[l]^-{\simeq}
} . \] Thus $\xi_\alpha(Y/K)$ doesn't depend on the isomorphism class of $Y/K$, so define $\xi_\alpha([Y/K]):=\xi_\alpha(Y/K)$. There is a homomorphism \begin{equation} \label{morphisme classifiant 2}
\xi^{\mathcal{F},d} : \left\lbrace \begin{array}{ccc}
 \h^d (\mathrm{B}G^\bullet,\mathcal{F}) & \to & \inv (G,\h^d (\star,\mathcal{F})) \\
 \alpha & \mapsto & \xi_\alpha 
 \end{array} \right. .
 \end{equation}

\begin{rem}
\label{remark functoriality xi}
The homomorphism $\xi^{\mathcal{F},d}$ is natural in $G$ : indeed, for a group homomorphism $\sigma : G \to G^\prime$ between smooth and connected linear algebraic $k$-groups, then writing $\xi$ for $\xi^{\mathcal{F},d}$ obtained for $G$ and $\xi^\prime$ for $\xi^{\mathcal{F},d}$ obtained for $G^\prime$, there is a commutative diagram \[ \xymatrix{
\h^d (\mathrm{B}(G^\prime)^\bullet,\mathcal{F}) \ar[r]^{\xi^\prime} \ar[d] & \inv (G^\prime,\h^d (\star,\mathcal{F})) \ar[d] \\
\h^d (\mathrm{B}G^\bullet,\mathcal{F}) \ar[r]_\xi & \inv (G,\h^d (\star,\mathcal{F}))
 } .\] The left vertical arrow is the one induced by the homomorphism $\mathrm{B}G^\bullet \to \mathrm{B} (G^\prime)^\bullet$ induced by $\sigma$ (see \ref{sous section torseur simplicial}). And the right vertical arrow is $\sigma^\ast$ (see Remark \ref{remarque foncto}).
\end{rem}

\paragraph{$\triangleright$ Comparing $\xi^{\mathbb{G}_m,2}$ and $\Lambda^\br$.}

Recall the versal $G$-torsor $\pi : E \to X$ associated to a faithful representation $G \hookrightarrow \mathrm{GL}_n=E$ (so $X=E/G$). Also for a $G$-torsor $Y \to Z$, we defined a simplicial scheme map $\phi_{Y/Z} : [Y|G]^\bullet \to \mathrm{B}G^\bullet$. In the case of $\pi$ this gives $\chi=\phi_{E/X} : [E|G]^\bullet \to \mathrm{B}G^\bullet$. Note that if $\sigma : Y/K \overset{\sim}{\to} x^{-1}(\pi)$ for some $x \in X(K)$, then $\phi_{Y/Z}$ factors as \[ [Y|G]^\bullet \to [E|G]^\bullet \overset{\chi}{\to} \mathrm{B}G^\bullet \] for $\psi_x : [Y|G]^\bullet \to [E|G]^\bullet$ obtained by quotienting the scheme maps \[ \begin{array}{ccc}
G^{n+1} \times_k Y & \to  & G^{n+1} \times_k E \\ 
((g_i),y) & \mapsto & ((g_i),\sigma(y))
\end{array}  \] with respect to the diagonal $G$-actions.

Now consider the homomorphism \[ \tau : \h^2(\mathrm{B} G^\bullet,\mathbb{G}_m) \to \h^2([E|G]^\bullet,\mathbb{G}_m) \simeq \h^2(X,\mathbb{G}_m)=\br(X)  \] induced by $\chi$.

We intend to show : 
\begin{prop}
\label{propo factorisation brauer}
The homomorphism $\tau$ factors through $\br(X)_\equi$ and one has a commutative diagram \[ \xymatrix{
\h^2(\mathrm{B} G^\bullet,\mathbb{G}_m) \ar[dr]_\xi \ar[r]^{\tau}  & \br(X)_{\text{bal}} \ar[d]^{\Lambda^\br}  \\
 & \inv(G,\br) 
}
 . \]
\end{prop}

\begin{proof}
The proof is inspired by the proof of \cite[Th. 3.4.]{MerkurjevBlinstein}

$\diamond$ First, note that once we know $\tau$ factors by $\br(X)_\equi$, it is easy to establish the commutativity of the diagram by inspecting the definitions of $\xi$ and $\Lambda^\br$.

$\diamond$ Let $\alpha \in \h^2(\mathrm{B} G^\bullet,\mathbb{G}_m)$ and let $Y/K$ be a $G$-torsor over the extension $K/k$. There exists $x \in X(K)$ such that $Y \simeq x^{-1}(\pi)$. The point $x$ can be used to define $\psi_x : [Y|G]^\bullet \to [E|G]^\bullet$ and taking cohomology we recover \[ x^\ast : \br(X) \simeq \h^2([E|G]^\bullet,\mathbb{G}_m) \overset{\psi_x^\ast}{\longrightarrow} \h^2([Y|G]^\bullet,\mathbb{G}_m) \simeq \br(K) . \]  Moreover, the composite of $\psi_x$ followed by $\chi : [E|G]^\bullet \to \mathrm{B} G^\bullet$ gives $\phi_{Y/K} : [Y|G]^\bullet \to \mathrm{B} G^\bullet$. Thus the following equalities holds : \[\xi_\alpha(Y/K)= \phi_{Y/K}^\ast(\alpha)= x^\ast(\chi^\ast(\alpha)) = x^\ast(\tau(\alpha)) \in \br(K) . \] Since $\xi_\alpha$ is a well defined invariant of $G$, the element $x^\ast(\alpha) \in \br(K)$ doesn't depend on the point $x$ for which $Y \simeq x^{-1}(\pi)$.

Consider $p_1$, $p_2 : (E \times_k E)/G \to X$ be the two projections. We want to show that $p_1^\ast(\tau(\alpha))=p_2^\ast(\tau(\alpha)) \in \br((E \times_k E )/G$. Write $\beta$ for $\tau(\alpha)$. The $G$-torsors $q_1^{-1}(\pi)$ and $q_2^{-1}(\pi)$ are isomorphic, where $q_i$ is the composite of the generic point $\spec \left( k(E \times_k E /G)\right) \to (E \times_k E)/G$ followed by $p_i$ ($i=1$, $2$). Thus, by what has been said, $q_1^\ast(\beta)=q_2^\ast(\beta)$ in $\br( k(E \times_k E /G))$. But $\br((E \times_k E /G)) \to \br(k(E \times_k E /G))$ is injective, whence $p_1^\ast(\beta)=p_2^\ast(\beta)$ in $\br((E \times_k E) /G))$, that is $\tau(\alpha)=\beta \in \br(X)_\equi$, as wanted.
\end{proof}

\paragraph{$\triangleright$ {Comments.}}

$\bullet$ Let $\mathcal{F}$ be an \'etale sheaf on $\underline{Sch/k}$ and let $d$ be an integer $\geqslant 0$. The construction of invariants via the simplicial method described above may not furnish \emph{all} degree $d$ cohomological invariants of $G$ with coefficients in $\mathcal{F}$ : the homomorphism $\xi^\mathcal{F}$ may not be onto.





\section{Full description of degree-2 invariants}
\label{section invariants degre 2}

In this section we state and show the main result of the paper. We show Theorem \ref{thm inv brauer} which implies Theorem \ref{cor inv brauer}. The latter extends \cite[Th. 2.4.]{MerkurjevBlinstein} to \emph{all} smooth and connected linear algebraic $k$-groups over an \emph{arbritrary} base field. Moreover, Theorem \ref{cor inv brauer} states a functoriality with respect to the implied group.

These results are actually a consequence of Proposition \ref{diagramme pour le thm} which comes from calculations and comparison of spectral sequences for $\mathrm{B}G^\bullet$ and $[E|G]^\bullet$, where $E \to X$ is a versal $G$-tosor with the good properties of Theorem \ref{iso blinstein merkurjev}.

		\subsection{Statement}

First let's state two isomorphisms whose proof is developped in Subsection \ref{sous section preuve} just above.

\begin{thm}
\label{thm inv brauer}
There exists two isomorphisms, natural in $G$ with $G$ any connected, smooth, linear algebraic $k$-group, \begin{gather}
\label{iso 1} 
 \mathrm{Ext}^1(G,\mathbb{G}_m) \overset{\sim}{\to} \h^2(\mathrm{B}G^\bullet,\mathbb{G}_m) / \h^2(k,\mathbb{G}_m) , \\
 \nonumber \\
\label{iso 2}
 \xi : \h^2(\mathrm{B}G^\bullet,\mathbb{G}_m) \overset{\sim}{\to} \inv(G,\br).
\end{gather} The morphism $\xi=\xi^{\mathbb{G}_m,2}$ is that from \ref{sous section invariants constructions} for the functor $\h^2(\star,\mathbb{G}_m)$.
\end{thm}

Putting together the morphisms \eqref{iso 1} and $\h^2(\mathrm{B}G^\bullet,\mathbb{G}_m) /\h^2(k,\mathbb{G}_m) \to \inv(G,\br)_0$ induced by \eqref{iso 2}, one gets the following description of the group of invariants with coefficients in the Brauer group :

\begin{thm}
\label{cor inv brauer}
There is an group isomorphism, natural in $G$ for $G$ any connected, smooth linear algebraic $k$-group, \begin{equation}
\label{iso inv brauer} \mathrm{Ext}^1(G,\mathbb{G}_m) \overset{\sim}{\to} \mathrm{Inv}(G,\mathrm{Br})_0 .
\end{equation}
\end{thm}

\paragraph{Sketch of the proof.}  Let's give a sketch of the proof to come. For $G$ as above, let $G \hookrightarrow \mathrm{GL}_n$ be a faithful representation. Write $E=\mathrm{GL}_n$ and $X=\mathrm{GL}_n/G$. Recall the constructions of invariants from the previous section : \begin{gather*}
	\xi=\xi^{\mathbb{G}_m,2} : \h^2(\mathrm{B}G^\bullet,\mathbb{G}_m) \to  \inv(G,\br) ,\\
	 \\
	\Lambda^\br : \br(X)_\equi  \to  \inv(G,\br) ,\\
\end{gather*} We have seen there is a factorization (Proposition \ref{propo factorisation brauer}) \[ \xymatrix{
\h^2(\mathrm{B} G^\bullet,\mathbb{G}_m) \ar[dr]_\xi \ar[r]^{\tau}  & \br(X)_{\text{bal}} \ar[d]^{\Lambda^\br}_{\simeq}  \\
 & \inv(G,\br) 
}
 ,\] that $\Lambda^\br$ is an isomorphism (lemma \ref{lemme iso brauer construction}) and $\xi$ is natural in $G$ (Remark \ref{remark functoriality xi}).

So the proof of \ref{cor inv brauer} relies on showing $\tau$ is an isomorphism and that $\ext(G,\mathbb{G}_m) \simeq \h^2(\mathrm{B}G^\bullet,\mathbb{G}_m) / \h^2(k,\mathbb{G}_m)$ functorialy in $G$. It is stated in Proposition \ref{diagramme pour le thm} below. In order to obtain these isomorphisms we will use the spectral sequence from Proposition \ref{Suite spectrale simpliciale} linking the cohomology of each piece of a simplicial scheme to its whole cohomology, for both the simplicial schemes $\mathrm{B}G^\bullet$ and $[X|G]^\bullet$.

		\subsection{Proof}
		\label{sous section preuve}

Let $G$ a smooth and connected linear algebraic $k$-group, and $\pi : E \to X$ be a versal torsor built from an embedding $G \hookrightarrow E=\mathrm{GL}_n$ (as seen in \ref{sous section torseurs classifiant}). We fix these data for all the subsection.

Consider the homomorphism \[ \epsilon^\ast : \h^2(\mathrm{B}G^{\bullet}, \mathbb{G}_m) {\to} \h^2(C_k \spec(k)^\bullet,\mathbb{G}_m) \underset{}{\simeq} \br(k) \] induced by the unit morphism of $G$ as a group scheme, $\epsilon : \spec(k) \to G$; and the homomrphism \[  \tau : \h^2(\mathrm{B}G^{\bullet}, \mathbb{G}_m) {\to} \h^2([E|G]^\bullet,\mathbb{G}_m) \underset{}{\simeq} \br(X)  \] where $\tau$ is the morphism induced by $[E|G]^\bullet \to \mathrm{B} G^\bullet$ from subsection \ref{sous section torseur simplicial}.

Theorem \ref{thm inv brauer} is a consequence of :

\begin{prop}
\label{diagramme pour le thm}
There is a commutative diagram with exact rows : \begin{equation}
\label{diagramme theoreme}
\xymatrix{ 
0 \ar[r] & \mathrm{Ext}^1(G,\mathbb{G}_m) \ar[r] \ar[d]^{=} & \h^2(\mathrm{B}G^{\bullet}, \mathbb{G}_m) \ar[r]^-{\epsilon^\ast} \ar[d]^{\tau} & \br(k) \ar[r] \ar[d] & 0 \\
0 \ar[r] & \mathrm{Ext}^1(G,\mathbb{G}_m) \ar[r] & \br(X) \ar[r]^{\pi^\ast} & \br(E) &
} .
\end{equation}

Moreover the first row is natural in $G$.
\end{prop}

Given Proposition \ref{diagramme pour le thm}, let's show Theorem \ref{thm inv brauer}. But first let's make some remarks on the decomposition of some cohomology groups with respect to the cohomology of the base field.

The particular $k$-point $\epsilon : \spec(k) \to G$ and $\epsilon_0 : \spec(k) \overset{\epsilon}{\to}  G \to \mathrm{GL}_n \to X$ enable to identify $\br(k)$ to direct factors of $\h^2(\mathrm{B}G_\bullet,\mathbb{G}_m)$ and of $\br(X)$. Then \[ \h^2(\mathrm{B}G_\bullet,\mathbb{G}_m) = \br(k) \oplus \h^2(\mathrm{B}G_\bullet,\mathbb{G}_m)_0 \: , \: \: \:  \: \: \: \:  \br(X) = \br(k) \oplus \br(X)_0  \] where \[ \h^2(\mathrm{B}G_\bullet,\mathbb{G}_m)_0 := \ker(\h^2(\mathrm{B}G_\bullet,\mathbb{G}_m) \overset{\epsilon^\ast}{\longrightarrow} \br(k) ) \] and likewise \[ \br(X)_0 : = \ker(\br(X) \overset{\epsilon_0^\ast}{\longrightarrow} \br(k)) . \] The map $\tau$ is the identity map on $\br(k)$ and respects the former decomposition. The same holds for $\xi$ \eqref{morphisme classifiant 2} and $\Lambda^{\br} : \br(X)_\equi \to \inv(G,\br)$ \eqref{construction equilibres} for the decomposition of $\inv(G,\br)$ into constant and normalized invariants.

\begin{proof}[Proof of Theorem \ref{thm inv brauer}]
Assume here that Proposition \ref{diagramme pour le thm} holds. First of all, take the isomorphism given by the first row of diagram \eqref{diagramme theoreme} to be the isomorphism \eqref{iso 1} $\mathrm{Ext}^1(G,\mathbb{G}_m) \overset{\sim}{\to} \h^2(\mathrm{B}G^\bullet,\mathbb{G}_m)_0 = \h^2(\mathrm{B}G^\bullet,\mathbb{G}_m) / \h^2(k,\mathbb{G}_m)$ of the theorem. It is natural in $G$ as wanted. Secondly, as it is said above, in order to have the isomorphism \eqref{iso 2} of the theorem, it is enough to see that $\xi$ induces an isomorphism $\h^2(\mathrm{B}G^\bullet,\mathbb{G}_m)_0 \to \inv(G,\br)_0$. This is what is done in the sequel.

Proposition \ref{diagramme pour le thm} implies that $\tau :\h^2(\mathrm{B}G^\bullet,\mathbb{G}_m)_0  \to \ker(\br(X) \to \br(E))$ is an isomorphism. According to Lemma \ref{lemme iso brauer construction}, one has $\Lambda^{\br} : (\br(X)_0)_\equi \simeq \inv(G,\br)_0$, so it remains to show that the subgroup $\ker(\br(X) \to \br(E)) \subseteq \br(X)$ is exactly $(\br(X)_0)_\equi$. TFrom Proposition \ref{propo factorisation brauer}, it follows that the inclusion \[ \ker(\br(X) \to \br(E)) \subseteq (\br(X)_0)_\equi \] holds. Regarding the reverse inclusion, let $\alpha \in (\br(X)_0)_\equi$. Consider the composition $h : \br(X) \to \br(E) \hookrightarrow \br(k(E))$, where the first arrow is induced by the versal torsor $\pi : E \to X$ and the second by the generic point $\spec(k(E)) \to E$ of $E$. The element $h(\alpha)$ is the value of of the invariant $\Lambda^{\br}(\alpha)$ in the torsor $E_\text{gen} \times_k k(E)$. But this torsor is the trivial $G$-torsor over $k(E)$, thus $h(\alpha)=0$, leading to $\alpha \in \ker(\br(X) \to \br(E))$. 
\end{proof}

Now let's show Proposition \ref{diagramme pour le thm} in details.

\begin{center}
\underline{\textit{Proof of Proposition \ref{diagramme pour le thm}}}
\end{center}

We begin by establishing separatly the two exact rows of Diagram \eqref{diagramme theoreme}. Actually, the process is the same for the two rows; only the last details differ.

\paragraph{Step 1 : An exact sequence.} Denote by $Y$ a $k$-scheme which is either $\spec(k)$ or $E=\mathrm{GL}_n$.

Apply Proposition \ref{Suite spectrale simpliciale} to the simplicial $k$-scheme $[Y|G]^\bullet$ (when $Y=\spec(k)$, it is just $\mathrm{B}G^\bullet$) : There is a convergent spectral sequence \begin{equation}
\label{suite convergente etape 2}
 E_1^{p,q}= \h^q([Y|G] ^{(p)},\mathbb{G}_m) \Rightarrow \h^{p+q}([Y|G]^\bullet,\mathbb{G}_m).
\end{equation}

The first page of this spectral sequence is \begin{center} 
\begin{tikzpicture} 
\node (B) at (0,0) {$\h^0(Y,\mathbb{G}_m)$};
\node (C) at (4,0) {$\h^0((G^2 \times_k Y)/G,\mathbb{G}_m)$};
\node (D) at (9,0) {$ \h^0((G^3 \times_k Y)/G,\mathbb{G}_m)$};
\node (E) at (12,0) {};

\node (B1) at (0,1) {$\h^1(Y,\mathbb{G}_m)$};
\node (C1) at (4,1) {$\h^1((G^2 \times_k Y)/G,\mathbb{G}_m)$};
\node (D1) at (9,1) {$ \h^1((G^3 \times_k Y)/G,\mathbb{G}_m)$};
\node (E1) at (12,1) {};

\node (B2) at (0,2) {$\h^2(Y,\mathbb{G}_m)$};
\node (C2) at (4,2) {$\h^2((G^2 \times_k Y)/G,\mathbb{G}_m)$};
\node (D2) at (9,2) {$ \h^2((G^3 \times_k Y)/G,\mathbb{G}_m)$};
\node (E2) at (12,2) {};

\draw[->] (D) -- (E);
\draw[->] (D1) -- (E1);
\draw[->] (D2) -- (E2);

\draw[->] (B.east) -- (C.west);
\draw[->] (C.east) -- (D.west);

\draw[->] (B1.east) -- (C1.west);
\draw[->] (C1.east) -- (D1.west);

\draw[->] (B2.east) -- (C2.west);
\draw[->] (C2.east) -- (D2.west);

\draw[thick,dashed,->] (-0.5,-0.5) -- (12,-0.5) node[below] {$p$};
\draw[thick,dashed] (0,-1) -- (B.south);
\draw[thick,dashed] (B.north) -- (B1.south);
\draw[thick,dashed] (B1.north) -- (B2.south);
\draw[thick,dashed,->] (B2.north) -- (0,3) node[right] {$q$};
\end{tikzpicture}
\end{center} where the differentials (from the $n$-th column to the $(n+1)$-th column) are $d=\sum_{i=0}^{n+1} (-1)^i \delta_i^\ast$ for $\delta_i :( G^{n+2} \times_k Y) /G \to (G^{n+1} \times_k Y)/G$ is induced by deleting the $i$-th coordinate of $G^{n+2}$ ($0 \leqslant i \leqslant n+1$ - the first coordinate is numbered $0$).

Following \cite[(1.3.1)]{Deligne_extensionscentrales}, let's consider the isomorphisms  \[ \begin{array}{ccc} 
(G^{n+1} \times_k Y) /G & \overset{\sim}{\to} & G^n \times_k Y \\
(g_0,\cdots,g_n,x) \, \mathrm{mod} \, G & \mapsto & (g_0 g_1^{-1},\cdots,g_{n-1} g_n^{-1},x \cdot g_0^{-1})
\end{array} .\] Thanks to these isomorphisms, the familly of $G^n\times_k Y$ can be seen as a $k$-simplicial scheme whose the face maps $\delta_i$ of $[Y|G]^\bullet$ are changed into \[ \delta_i^\prime =\delta_i^{n,\prime} : \left\lbrace \begin{array}{cccl}
G^n \times_k Y & \to & G^{n-1} \times_k Y &  \\
(g_1,\cdots,g_n,x) & \mapsto & (g_2,\cdots,g_n,x \cdot g_1) & \text{si } i=0 \\
& & (g_1,\cdots,g_{i-1},g_i g_{i+1},g_{i+2},\cdots,g_n,x) & \text{si }  0 < i <n \\
& & (g_1,\cdots,g_{n-1},x) & \text{si } i=n  
\end{array} \right. .\]

For clarity, let's precise what are $\delta_i^\prime$ in the limit cases $i=1$ and $i=n-1$ : if $n=2$, $\delta_1^\prime(g_1,g_2)=g_1g_2$, and if $n \geq 2$, one has \begin{align*}
\delta_1^\prime (g_1,\cdots,g_n,x) & = (g_1g_2,g_3,\cdots,g_n,x) \\
\delta_{n-1}^\prime (g_1,\cdots,g_n,x) & = (g_1,\cdots,g_{n-2},g_{n-1}g_n,x).
\end{align*} The above first page of the spectral sequence then becomes \begin{center}
\begin{tikzpicture}
\node (A) at (-2,0) {$0$};
\node (B) at (0,0) {$\h^0(Y,\mathbb{G}_m)$};
\node (C) at (3,0) {$\h^0(G\times_k Y,\mathbb{G}_m)$};
\node (D) at (6.5,0) {$ \h^0(G^2 \times_k Y,\mathbb{G}_m)$};
\node (E) at (9,0) {$\cdots$};
\node (A1) at (-2,1) {$0$};
\node (B1) at (0,1) {$\h^1(Y,\mathbb{G}_m)$};
\node (C1) at (3,1) {$\h^1(G \times_k Y,\mathbb{G}_m)$};
\node (D1) at (6.5,1) {$ \h^1(G^2 \times_k Y,\mathbb{G}_m)$};
\node (E1) at (9,1) {$\cdots$};
\node (A2) at (-2,2) {$0$};
\node (B2) at (0,2) {$\h^2(Y,\mathbb{G}_m)$};
\node (C2) at (3,2) {$\h^2(G \times_k Y,\mathbb{G}_m)$};
\node (D2) at (6.5,2) {$ \h^2(G^2 \times_k Y,\mathbb{G}_m)$};
\node (E2) at (9,2) {$\cdots$};
\draw[->] (A.east) -- (B.west);
\draw[->] (B.east) -- (C.west);
\draw[->] (C.east) -- (D.west);
\draw[->] (D.east) -- (E.west);
\draw[->] (A1.east) -- (B1.west);
\draw[->] (B1.east) -- (C1.west);
\draw[->] (C1.east) -- (D1.west);
\draw[->] (D1.east) -- (E1.west);
\draw[->] (A2.east) -- (B2.west);
\draw[->] (B2.east) -- (C2.west);
\draw[->] (C2.east) -- (D2.west);
\draw[->] (D2.east) -- (E2.west);
\draw[thick,dashed,->] (-2.5,-0.5) -- (10,-0.5) node[below] {$p$};
\draw[thick,dashed] (0,-1) -- (B.south);
\draw[thick,dashed] (B.north) -- (B1.south);
\draw[thick,dashed] (B1.north) -- (B2.south);
\draw[thick,dashed,->] (B2.north) --  (0,3) node[right] {$q$};
\end{tikzpicture}
\end{center} where the differentials are $d=\sum (-1)^i (\delta^\prime_i)^\ast$.

\begin{lemma}
\label{cohomologie nulle etape 2}
The complex $(\h^0(G^n \times_k Y,\mathbb{G}_m))_{n \in \mathbb{N}}=(k[G^n \times_k Y]^\ast)_{n \in \mathbb{N}}$ has zero cohomology in degree $\geqslant 2$.
\end{lemma}

\begin{proof} For a $k$-group scheme $H$, we write $\widehat{H}$ for the character group of $H$ : \[ \widehat{H}=\mathrm{Hom}_{k-\text{group}}(H,\mathbb{G}_m) . \]

Recall $Y=\spec(k)$ or $Y=\mathrm{GL}_n$. Since $G$ and $Y$ are smooth and connected $k$-groups, Rosenlicht lemma says that the complex in stake can be identified to $(k^\times \oplus \widehat{G^n \times Y}(k))_n=( k^\times \oplus  \widehat{G}(k)^n \oplus \widehat{Y}(k))_n$ which is the sum of the complex $C_1:=(k^\times)_n$ with differentials \[ d_1^n = \left\lbrace \begin{array}{l}
0 \text{ si } n \text{ pair} \\
\mathrm{Id}_{k^\times} \text{ si } n \text{ impair}
\end{array} \right. \] and of the complex $C_2:=(\widehat{G}(k)^n \oplus \widehat{Y}(k))_n$ with differentials \[ d_2^n=\sum_{i=0}^{n+1} (-1)^i ({\delta}_i^{n+1,\prime})^\star \] where $({\delta}_i^{n+1,\prime})^\star : \widehat{G}^{n}(k) \oplus \widehat{Y}(k) \to \widehat{G}^{n+1}(k) \oplus \widehat{Y}(k)$ is the homomorphism on character groups induced by $\delta_i^{n+1,\prime}$. \\ Complex $C_1$ has clearly zero cohomology. For $C_2$, we are in the same situation as in the proof of \cite[Lem. 6.12.]{Sansuc_Brauer}. The differentials are \begin{align*}
d_2^{2i}(\phi_1,\cdots,\phi_{2i},\psi) & =(\psi_{|G} - \phi_1, 0,\cdots, \underset{2j+1-\text{th position}, \, j \geqslant 1}{\underbrace{\phi_{2j}-\phi_{2j+1}}},\underset{2j+2-\text{th position}}{\underbrace{0}}, \cdots,\phi_{2i},0 ) \\
d_2^{2i+1}(\phi_1,\cdots,\phi_{2i+1},\psi) & =(\psi_{|G}, \phi_2,\phi_2,\cdots, \underset{2j-\text{th position}, \, j \geqslant 1}{\underbrace{\phi_{2j}}},\underset{2j+1-\text{th position}}{\underbrace{\phi_{2j}}}, \cdots,\phi_{2i},\phi_{2i},\psi )
\end{align*} for all $i \geqslant 0$ and where $\phi_j$ describes $\widehat{G}(k)$, and $\psi$ describes $\widehat{Y}(k)$. Following the proof of  \cite[Lem. 6.12.]{Sansuc_Brauer}, the maps, $n \geqslant 2$, \[ s_n : \left\lbrace \begin{array}{ccc}
 \widehat{G}(k)^n \oplus \widehat{Y}(k) & \to & \widehat{G}(k)^{n-1} \oplus \widehat{Y}(k) \\
 (\phi_1,\cdots,\phi_n,\psi) & \mapsto & (-\phi_1,\phi_3,\cdots,\phi_n,\psi)
\end{array} \right. \] defines a homotopy in degree $\geqslant 2$ between the idendity map of $C_2$ and the zero map, that is : for all $n \geqslant 2$, one has $d_2^{n-1} \circ s_n + s_{n+1} \circ d_2^{n} = \mathrm{id}_{\widehat{G}(k)^n \oplus \widehat{Y}(k)}$, thus the cohomology of $C_2$ is zero in degree $\geqslant 2$. 
\end{proof}

\begin{lemma}
\label{identification complexe etape 2}
The projections $G^n \times_k E \to G^n$, $((g_i),x) \mapsto (g_i)$ identify together the complexes $(\h^1(G^n,\mathbb{G}_m))_{n \in \mathbb{N}}$ and $(\h^1(G^n \times_k E,\mathbb{G}_m))_{n \in \mathbb{N}}$.
\end{lemma}

\begin{proof}
One has an identification $\h^1(G^n \times_k Y,\mathbb{G}_m)=\pic(G^n \times_k Y)$. According to \cite[Lem. 11]{ColliotSansuc_Requivalence} (or \cite[Lem. 6.6.]{Sansuc_Brauer}) (note that $E=\mathrm{GL}_n$ is indeed rational over $k$), the following projections $\mu_n : G^n \times_k E \to G^n$ and $\nu_n : G^n \times_k E \to E$ induce an isomorphism $\mu_n^\ast \oplus \nu_n^\ast : \pic(G^n) \oplus \pic(E) \overset{\sim}{\to} \pic(G^n \times_k E )$. But $\pic(\mathrm{GL}_n)=0$, so the previous isomorphism is $\mu_n^\ast : \pic(G^n) \overset{\sim}{\to} \pic(G^n \times_k E )$, whose inverse is $(i_n^x)^\ast$ which is induced by any embedding $i_n^x : G^n \hookrightarrow G^n \times_k E$ built from any $k$-point $x$ of $E$. Thus, the maps $(\delta_i^\prime)^\ast : \pic(G^n \times_k E ) \to \pic(G^{n-1} \times_k E)$ obtained for $Y=E$ becomes $(\delta_i^\prime)^\ast : \pic(G^n) \to \pic(G^{n-1})$ obtained for $Y=\spec(k)$. Whence the lemma.
\end{proof}

Thanks to Lemmas \ref{cohomologie nulle etape 2} and \ref{identification complexe etape 2}, the second page of the spectral sequence \eqref{suite convergente etape 2} becomes (we write $\br$ for $\h^2(-,\mathbb{G}_m)$) \begin{equation}
\label{page 2 suite etape 2}
\end{equation} 
\begin{center}
\begin{tikzpicture}
\node (B) at (0,0) {$\h^0(k[G^\bullet \times_k Y]^\times)$};
\node (C) at (4,0) {$\h^1(k[G^\bullet \times_k Y]^\times)$};
\node (D) at (8,0) {$0$};
\node (E) at (11,0) {$0$};

\node (A1) at (-3,1.5) {$0$};
\node (B1) at (0,1.5) {$0$};
\node (C1) at (4,1.5) {{\small $\ker(\pic(G) \to \pic(G^2))$}};
\node (D1) at (8,1.5) {$ E_2^{2,1}$};

\node (A2) at (-3,3) {$0$};
\node (B2) at (0,3) {{\small $\ker(\br(Y) \to \br(G \times_k Y))$}};
\node (C2) at (4,3) {$E_2^{1,2}$};
\node (D2) at (8,3) {$ E_2^{2,2}$};

\draw[->] (A1) -- (C);
\draw[->] (B1) -- (D);
\draw[->] (C1) -- (E);
\draw[->] (A2) -- (C1);
\draw[->] (B2) -- (D1);

\draw[thick,dashed] (-2.5,0) -- (B);
\draw[thick,dashed] (B) -- (C);
\draw[thick,dashed] (C) --(D);
\draw[thick,dashed] (D) -- (E);
\draw[thick,dashed,->] (E) -- (11.5,0) node[below] {$p$};
\draw[thick,dashed] (0,-1) -- (B.south);
\draw[thick,dashed] (B.north) -- (B1.south);
\draw[thick,dashed] (B1.north) -- (B2.south);
\draw[thick,dashed,->] (B2.north) -- (0,4) node[right] {$q$};
\end{tikzpicture}.
\end{center} Thus, using the usual notations for spectral sequences and their limit terms (\cite[Chap. 5]{Weibel_homalg}), \begin{gather*} E_\infty^{0,2} =E_3^{0,2}=\ker \left( \ker(\br(Y) \to \br(G \times_k Y)) \to E_2^{2,1} \right), \\  E_\infty^{1,1}=E_2^{1,1}=\ker(\pic(G) \to \pic(G^2)) \end{gather*} and $E_\infty^{2,0}=0$. Also, since the spectral sequence \eqref{suite convergente etape 2} is convergent, one finds an exact sequence, natural in faithful representation $G \hookrightarrow \mathrm{GL}_n$ : \begin{equation}
 \label{suite exacte partielle etape 2} \begin{split}
 0 \to \ker\left( \pic(G) \to \pic(G^2) \right) & \to \h^2([Y|G]^\bullet,\mathbb{G}_m)  \to  \\
 & \ker \left( \ker(\br(Y) \to \br(G \times_k Y)) \to E_2^{2,1} \right) \to 0 . 
 \end{split}
 \end{equation}
 
The sequence \eqref{suite exacte partielle etape 2} will be study for $Y=\spec(k)$ and also for $Y=E$ in the following steps $2$ and $3$ respectiveley.

\paragraph{Step 2 : The exact sequence for $Y=\spec(k)$.} 

The sequence \eqref{suite exacte partielle etape 2} for $Y=E$ is : \begin{equation}
\label{suite exacte partielle etape 22} \begin{split} 0 \to \ker\left( \pic(G) \to \pic(G^2) \right) & \to \h^2(\mathrm{B}G^\bullet,\mathbb{G}_m)  \to \\
& \ker \left( \ker(\br(k) \to \br(G)) \to E_2^{2,1} \right) \to 0 .
\end{split}
\end{equation} By naturality of \eqref{suite exacte partielle etape 22} with respect to the unit morphism $\epsilon : \spec(k) \to G$ it comes that the second half of \eqref{suite exacte partielle etape 22} is actually the projection $\h^2(\mathrm{B}G^\bullet,\mathbb{G}_m)  = \br(k) \oplus \h^2(\mathrm{B}G^\bullet,\mathbb{G}_m) _0 \to \br(k) \to 0$.
 
It remains to state that $\ker (\pic(G) \to \pic(G^2))$ is actually the kernel of \[ (\delta_0^\prime)^\ast-(\delta_1^\prime)^\ast+(\delta_2^\prime)^\ast : \pic(G) \to \pic (G^2) .\] But this homomorphism is just $p_1 ^\ast + p_2^ \ast-m^\ast$ where $p_i$ is the projection of $G^2$ onto the $i$-th factor, $i=1$, $2$, and $m$ is the multiplication of $G$. So the kernel of $p_1 ^\ast + p_2^ \ast-m^\ast$ is $\mathrm{Ext}^1(G,\mathbb{G}_m)$, the isomorphism between the two groups being natural in $G$. The latter fact can be shown using exactly the same proof of \cite[Th. 5.6.]{Colliot_resolutions} : Colliot-Th\'el\`ene shows (in particular) that for a smooth and connected linear algebraic $k$-group $H$, assumed reductive when the characteristic of $k$ is positive, then $\ext(H,\mathbb{G}_m)=\pic(H)$, but the proof actually shows that for any smooth and connected linear algebraic $k$-group $H$, every element of $\mathrm{Ker}(p_1 ^\ast + p_2^ \ast-m^\ast)$ comes from an extension of $H$ by $\mathbb{G}_m$.
 
To sum up, one has a exact sequence, functorial with respect to the smooth and connected linear algebraic $k$-group $G$, \[ 0 \to \mathrm{Ext}^1(G,\mathbb{G}_m) \to \h^2(\mathrm{B}G_\bullet,\mathbb{G}_m)  \overset{\epsilon^\ast}{\to} \br(k)  \to 0  \] where $\epsilon : \spec(k) \to G$ is the unit morphism of $G$. This concludes Step $2$ of the proof of Proposition \ref{diagramme pour le thm}.

\paragraph{Step 3 : The exact sequence for $Y=E$.} 

Let's use again the sequence \eqref{suite exacte partielle etape 2} with $Y=E$ : \begin{equation}
\label{suite exacte partielle etape 3} \begin{split} 0 \to \ker\left( \pic(G) \to \pic(G^2) \right) & \to  \h^2([E|G]^\bullet,\mathbb{G}_m)  \to \\
&  \ker \left( \ker(\br(E) \to \br(G \times_k E)) \to E_2^{2,1} \right) \to 0 . 
\end{split}
\end{equation} But $\h^2([E|G]^\bullet,\mathbb{G}_m) \simeq \h^2(X,\mathbb{G}_m)=\br(X)$ (Proposition \ref{iso cohomologie simplicial}) and recall from Step $2$ that \[ \ker\left( \pic(G) \to \pic(G^2) \right) \simeq \mathrm{Ext}^1(G,\mathbb{G}_m), \] so the exact sequence \eqref{suite exacte partielle etape 3} induces the exact sequence : \begin{equation}
 \label{suite exacte etape 3} 
 0 \to \mathrm{Ext}^1(G,\mathbb{G}_m) \to \br(X)  \to  \br(E) .
 \end{equation}

In order to completly establish the second row in Proposition \ref{diagramme pour le thm}, it remains to explain why the arrow $\br(X)  \to  \br(E)$ from \eqref{suite exacte etape 3} is indeed $\pi^\ast$, where $\pi$ is the projection $E \to X$. This fact is directly implied by the comparison of \eqref{suite exacte partielle etape 3} for $[E|G]^\bullet$ and the same sequence for  $C_k X^\bullet$, the comparison being obtained by that of the spectral sequences \eqref{suite convergente etape 2}
 to \begin{equation*}
\label{suite convergente etape 3 X}
( E^\prime)^{p,q}= \h^q(X,\mathbb{G}_m) \Rightarrow \h^{p+q}(C_k X^\bullet,\mathbb{G}_m) .
 \end{equation*} with respect to the projection  $\Pi_\bullet : [E|G]^\bullet \to C_k X^\bullet$, $((g_i),e) \to \pi(e)$
 
\paragraph{Step 4 : Commutativity of Diagram \eqref{diagramme theoreme}.}

The commutativity of \eqref{diagramme theoreme} is given by the spectral sequence morphism \[ \begin{array}{ccc}
\h^q(\mathrm{B} G^p,\mathbb{G}_m) & \Rightarrow&  \h^{p+q}(\mathrm{B}G^\bullet,\mathbb{G}_m) \\
\downarrow & & \downarrow \\
\h^q([X|G]^p,\mathbb{G}_m) & \Rightarrow & \h^{p+q}([X|G]^\bullet,\mathbb{G}_m)
\end{array} .\] Indeed, the spectral sequence \eqref{Suite spectrale simpliciale} is natural with respect to the morphism $[X|G]^\bullet \to \mathrm{B}G^\bullet$ from \ref{sous section torseurs classifiant}.

\begin{center}
\underline{\textit{End of the proof for Proposition \ref{diagramme pour le thm}}}
\end{center}

\section{Applications}
\label{sect consequences}

In this section we discuss Theorem \ref{cor inv brauer} and derive some consequences from it.

		\subsection{Comments}
			\label{sous section cas unipotent}

\paragraph{The reductive case.}
In the case where the group $G/k$ is (connected) reductive, then combining \cite[Lem. 6.6.]{Sansuc_Brauer} and \cite[Th. 5.6]{Colliot_resolutions} it comes that the canonical homomorphism $\ext(G,\mathbb{G}_m) \to \pic(G)$ is an isomorphism and we recover \cite[Th. 2.4.]{MerkurjevBlinstein}. Moreover, for semisimple $G$, $\pic(G)=\ext(G,\mathbb{G}_m)\simeq \widehat{C}$ where ${C}$ is the kernel of the universal cover of $G$ and $\widehat{C}$ the group of homomorphisms $C \to \mathbb{G}_m$ defined over $k$.

\paragraph{New calculations.} We only announce here a small result determining the group of invariants with value in the Brauer group for some groups that are not reductive over a imperfect field, namely \emph{pseudo-semisimple} groups (see \cite{CGP}). The presentation of the calculations is to be published separatly because too long and linked to other subjects to be developped here. By determining the groups of extensions by $\mathbb{G}_m$ of certain pseudo-semisimple groups, we derive their groups of $\br$-invariants.

Assume $p$, the characteristic of $k$, is positive and denote by $\mathrm{R}_{K/k}$ the Weil restriction functor through the finite field extension $K/k$.

\begin{prop}
\label{prop inv groupes pss}
Let $k^\prime/k$ be a finite and purely inseparable field extension. Call $h$ the least integer $\geqslant 0$ such that $(k^\prime)^{p^h} \subseteq k$. Let $G^\prime$ be a semisimple, simply connected $k^\prime$-group and $\mu^\prime \subseteq G^\prime$ be a central subgroup. Letting $G:=\res(G^\prime)/\res(\mu^\prime)$, there is an isomorphism \[ \inv(G,\br)_0 \simeq p^h \widehat{\mu^\prime} \] where $\widehat{\mu^\prime}$ is the group of characters of $\mu^\prime$ defined over $k^\prime$. 
\end{prop}

In particular, let $m$ be a positive integer and let $k^\prime/k$ and $h$ be as in the proposition. There is an isomorphism : \[ \inv \left( \res(\mathrm{SL}_{p^m,k^\prime})/\res(\mu_{p^m,k^\prime}), \br \right) \simeq p^h \cdot \mathbb{Z}/p^m , \] where \begin{enumerate}
\item $\mathrm{SL}_{p^m,k^\prime}$ is the special linear group of size $p^m$ over $k^\prime$;
\item $\mu_{p^m,k^\prime}$ is the group scheme of $p^m$-rooth of unity inside $\mathrm{SL}_{p^m,k^\prime}$.
\end{enumerate}  One sees that if $h \geqslant m$, then the group of invariants is trivial.

\paragraph{The case of unipotent groups.}
Let's talk about the Picard group and the group of extensions by $\mathbb{G}_m$ in the case of smooth and connected unipotent groups. We only state some results from other authors. Together with Theorem \ref{cor inv brauer}, these results give new information on the group of invariants with value in the Brauer group for unipotent groups, especially over imperfect fields. We see that in this case, the group of invariants is very different from the one for reductive groups or for groups over a perfect field.

Indeed, over a perfect field $k$, smooth and connected unipotent $k$-groups have composition series whose quotients are addidive group $\mathbb{G}_a^n$. The Picard groups of such groups are thus trivial and they have no non constant invariant. Things are more exotic over imperfect base fields.

Achet (\cite[Th. 2.4]{Achet_picard}) and Totaro (\cite[Lem. 9.3 \& 9.4]{Totaro_pseudoabelian}) began by showing that the Picard groups are non zero :

\begin{prop}
Let $G$ be a non trivial $k$-form of $\mathbb{G}_a$. Then $\pic(G) \neq 0$. Moreover, if $k$ is separably closed then certainly $\mathrm{Ext}^1(G,\mathbb{G}_m) \neq 0$.
\end{prop}

\begin{ex}
The first examples of non trivial form of $\mathbb{G}_a$ over a imperfect base field are subgroups $\{ (x,y) \, | \, y^p=x+ax^p \}$ of $\mathbb{G}_a^2$ where $a$ is any element in $k \setminus k^p$.
\end{ex}

Rosengarten went further, showing the following propositions. The group $\mathrm{Ext}^1(G,\mathbb{G}_m)$ can be infinite :

\begin{prop}[{\cite[Prop. 5.7]{Rosengarten_picard}}]
\label{prop pic unipotent infini}
Suppose $k$ is imperfect and separably closed (\textit{i.e.} $k=k_s$). If the characteristic of $k$ is not $2$, then for any $k$-form $G$ of the additive group $\mathbb{G}_a$, the group $\mathrm{Ext}^1(G,\mathbb{G}_m)$ is finite if and only if $G \simeq \mathbb{G}_a$. In particular, for all $a \in k \setminus k^p$, writing $G=\{ (x,y) \, | \, y^p=x+ax^p \} \subset \mathbb{G}_a^2$, the group $\ext(G,\mathbb{G}_m)$ is infinite.
\end{prop}

Moreover the natural homomorphism $\mathrm{Ext}^1(G,\mathbb{G}_m) \to \pic(G)$ is not onto in general for smooth and connected unipotent groups :

\begin{prop}[{\cite[Prop. 5.10]{Rosengarten_picard}}]
\label{non iso extension picard}
Let $G$ be a $k$-form of $\mathbb{G}_a$. Suppose the characteristic of $k$ is different from $2$. Then the natural homomorphism $\mathrm{Ext}^1(G,\mathbb{G}_m) \to \pic(G)$ is an isomorphism if and only if  $G \simeq \mathbb{G}_a$.
\end{prop}

Nevertheless there is the following restriction on the Picard group (and hence on the group of invariants with value in the Brauer group) :

\begin{prop}
\label{prop annulation p}
Let $G$ be a smooth and connected unipotent group over a field $k$. Then $\pic(G)$ is $q$-torsion for some power $q$ of the characteristic of $k$.
\end{prop}

\begin{proof}
There exists a finite purely inseparable extension $k^\prime/k$ such that $G_{k^\prime}$ is split as a unipotent group (that is $G_{k^\prime}$ has a composition series whose quotients are some $\mathbb{G}_a^r$). Thus $\pic(G_{k^\prime})=0$. But by \cite[Lem. 2.4]{Brion_linearization}, the kernel of $\pic(G) \to \pic(G_{k^\prime})$ obtained by scalar extension, is annihilated by the degree of $k^\prime/k$, which is a power of the characteristic of $k$.
\end{proof}

Since the Picard group of a smooth and connected unipotent group may be infinite, the previous proposition implies that when the Picard group is indeed infinite, then it can't be finitely generated.

\vspace{\baselineskip}

Due to the results stated above, it appears that one can't replace $\mathrm{Ext}^1(\star,\mathbb{G}_m)$ by $\pic(\star)$ to describe the group of invariants in Theorem \ref{cor inv brauer}. Also, the theorem implies that there exists smooth and connected linear algebraic group with infinite group of invariants.

\begin{ex}
Assume $k$ is separably closed and let $G$ be a $k$-form of $\mathbb{G}_a$ such that $\mathrm{Ext}^1(G,\mathbb{G}_m)$ is infinite. According to \cite[Cor. 95]{Totaro_pseudoabelian} there exists a commutative pseuso-reductive group $E$ which is an extension of $G$ by $\mathbb{G}_m$. Doing as in \cite[Ex. 5.14]{Rosengarten_picard} but working with $\mathrm{Ext}^1(\star,\mathbb{G}_m)$ instead of $\pic(\star)$, one shows that $\mathrm{Ext}^1(E,\mathbb{G}_m)$ is also infinite. In that case $\inv(E,\br)_0$ is infinite and $E$ is not unipotent.
\end{ex}

			\subsection{Invariants given by connecting morphisms.}

Besides what was done in section \ref{section presentation}, there is an elementary and classical receipe to define invariants of an algebraic group $G$ with values in the Brauer group. It is showed here that this elementary receipe gives all the invariants of $G$ in $\br$.

First of all let's recall the construction as it can be found in \cite[\S 31.B]{KMRT}. It seems that the construction has no name in the literature and we suggest to name it the \emph{connection construction}, which gives the \emph{connection morphism} (since it is based on connection morphisms of cohomology long exact sequences arising from short exact sequences of some sheaves).

\paragraph{\textit{Connection construction.}} 
Let $G$ be a smooth and connected algebraic linear $k$-group. Let $\mathcal{E}$ be an element in the group extension $\mathrm{Ext}^1(G,\mathbb{G}_m)$ represented by an extension \[ (E) \: \: \: 1 \to \mathbb{G}_m \to G^\prime \to G \to 1 . \] For all field extension $K/k$ one has the extension $E_K$ obtained by scalar extension from $E$. Since $G$ is connected, $E_K$ has to be a central extension. Thus for all $K/k$, $E_K$ induces a long exact sequence of pointed sets in Galois cohomology \[ 0 \to k^\times \to \cdots \to \h^1(K,G) \overset{\delta^E_K}{\longrightarrow} \h^2(K,\mathbb{G}_m)  \]  and $\delta^E_K$ is a map natural in $G$ and $K$ which does not depend on the representative $E$ of $\mathcal{E}$. All the maps $\delta^E_K$ for all $K/k$ determines a cohomological invariant of $G$ with values in the Brauer group. It is even a normalized invariant, for the arrows $\delta^E_K$ map the class of the trivial torsor to $0$. Write $\delta^\mathcal{E}$ for the invariant defined from $\mathcal{E} \in \mathrm{Ext}^1(G,\mathbb{G}_m)$, and $\delta=\delta_G$ for the map $\mathcal{E} \to \delta^\mathcal{E}$, $\mathrm{Ext}^1(G,\mathbb{G}_m) \to \inv(G,\br)_0$. 

The homomorphism $\delta$ is indeed a group homomorphism. We propose to name $\delta=\delta_G$ the \emph{connection morphism (of $G$)}.

\paragraph{Connection invariants.}

Here is the general result regarding the connection construction.

\begin{prop}
\label{prop inv classiques}
Let $G$ be a smooth and connected linear algebraic $k$-group. Then the homomorphism $\delta : \mathrm{Ext}^1(G,\mathbb{G}_m) \to \inv(G,\br)_0$ given by the connection construction is bijective.
\end{prop}

Rather than trying to determine explicitly wether the isomorphism of Theorem \ref{cor inv brauer} is actually $\delta$ (this requires to calculate explicitly many homomorphisms hidden in spectral sequences) we bypass the difficulties with the following proof based on the functoriality in Theorem \ref{cor inv brauer}.

\begin{proof}
$\diamond$ Let's first show that $\delta : \mathrm{Ext}^1(G,\mathbb{G}_m) \to \inv(G,\br)_0$ is one-to-one. \\ Let $E$ be a group extension \[ 1 \to \mathbb{G}_m \to G^\prime \to G \to 1 \] in the isomorphism class $\mathcal{E} \in \mathrm{Ext}^1(G,\mathbb{G}_m)$, and assume $\delta^\mathcal{E}$ is the invariant constantly equal to $0 \in \br$.

Because of the long exact sequence in Galois cohomology attached to $E$, the fact that $\delta^\mathcal{E}$ is trivial implies 
that for all field extension $K/k$, the maps $\h^1(K,G^\prime) \to \h^1(K,G)$ are onto. Thus the homomorphism $\inv(G,\br)_0 \to \inv(G^\prime,\br)_0$ is one-to-one. So the canonical homomorphism $\mathrm{Ext}^1(G,\mathbb{G}_m) \to \mathrm{Ext}^1(G^\prime,\mathbb{G}_m)$ is also one-to-one by Theorem \ref{cor inv brauer}. Since $\mathcal{E} \in \ker(\mathrm{Ext}^1(G,\mathbb{G}_m) \to \mathrm{Ext}^1(G^\prime,\mathbb{G}_m))$, one sees that $\mathcal{E}$ is the neutral element of $\mathrm{Ext}^1(G,\mathbb{G}_m)$.

$\diamond$ Now let's show that $\delta : \mathrm{Ext}^1(G,\mathbb{G}_m) \to \inv(G,\br)_0$ is onto. The composite of the homomorphism $\delta$ followed by the isomorphism $\inv(G,\br)_0 \simeq \ext(G,\mathbb{G}_m)$ from Theorem \ref{cor inv brauer} is a one-to-one homomorphism $\sigma_G : \ext(G,\mathbb{G}_m) \to \ext(G,\mathbb{G}_m)$, and $\sigma_G$ is natural in $G$. Let $\mathcal{E} \in \ext(G,\mathbb{G}_m)$ be the isomorphism class of the group extension \[ (E) \: \: \: 1 \to \mathbb{G}_m \to H \to {G} \to 1 . \] We show that the subgroup $\langle \mathcal{E} \rangle$ of $\ext(G,\mathbb{G}_m)$ generated by $\mathcal{E}$ is map to itself through $\sigma_G$. On one hand, the kernel of the pullback map $\tau : \ext(G,\mathbb{G}_m) \to \ext(H,\mathbb{G}_m)$ is exactly $\langle \mathcal{E} \rangle$. On the other hand, since $\sigma_G$ is natural in $G$, the following diagram is commutative : \[ \xymatrix{
\mathcal{E}^\prime \in \ext(G,\mathbb{G}_m) \ar[d]_{\sigma_G} \ar[r]^-\tau & \ext(H,\mathbb{G}_m) \ar[d]^{\sigma_H} \\
\sigma(\mathcal{E}^\prime) \in \ext(G,\mathbb{G}_m) \ar[r]_-\tau & \ext(H,\mathbb{G}_m)
} . \] The homomorphism $\sigma_H$ is one-to-one, so for all $\mathcal{E}^\prime \in \ext(G,\mathbb{G}_m)$ one has the equivalence : $\tau(\mathcal{E}^\prime) =0 \Leftrightarrow \tau \circ \sigma_G(\mathcal{E}^\prime) =0$, which means $\sigma_G(\mathcal{E}^\prime) \in \langle \mathcal{E} \rangle$ if and only if $\mathcal{E}^\prime \in \langle \mathcal{E} \rangle$. The facts that $\sigma_G$ is injective and $\langle \mathcal{E} \rangle$ is finite imply $\sigma_H$ induces an isomorphism of $\langle \mathcal{E} \rangle$ onto itself. Thus one can find a preimage of $\mathcal{E}$ through $\sigma_G$ and so a preimage of $\mathcal{E}$ through $\delta$.
\end{proof}

\bibliographystyle{alpha}
\bibliography{/Users/alexlourdeaux/Documents/Mathematiques/Perso/Bibliographie.bib}

\end{document}